\documentclass{amsart}
\usepackage{amssymb}
\usepackage{latexsym}
\usepackage{amsthm}
\usepackage{amscd}
\theoremstyle{plain}

\newtheorem{thm}{Theorem}[section]

\newtheorem{lem}[thm]{Lemma}

\newtheorem{defn}{Definition}[section]
\newtheorem{rem}{Remark}[section]

\newtheorem{exam}[thm]{Example}

\numberwithin{equation}{section}

\newcommand{\Hom}{\mathrm{Hom}}
\newcommand{\lfl}{\left\lfloor}
\newcommand{\rfl}{\right\rfloor}

\newcommand{\ZZ}{\mathbb Z}
\newcommand{\CC}{\mathbb C}

\newcommand{\RR}{\mathbb R}

\newcommand{\HH}{\mathbb H}

\newcommand{\csp}{\null\hskip 20pt}

\newcommand{\ccccsp}{\null\hskip 160pt}

\newcommand{\vsp}{\vskip 0.1cm}
\newcommand{\vvsp}{\vskip 0.2cm}

\begin{document}

\title[A NEW BASIS FOR THE SPACE OF MODULAR FORMS]
{A new basis for the space of modular forms}
\author{Shinji Fukuhara}
\thanks{The author wishes to thank Professors Noriko Yui and Yifan Yang
for the helpful comments.}
\thanks{\it{Address.} \rm{Department of Mathematics,
  Tsuda College, Tsuda-machi 2-1-1, \\
  Kodaira-shi, Tokyo 187-8577, Japan
  (e-mail: fukuhara@tsuda.ac.jp).}
}

\begin{abstract}
Let $G_{2n}$ be the Eisenstein series of weight $2n$
for the full modular group $\Gamma=SL_2(\ZZ)$.
It is well-known that the space $M_{2k}$ of modular forms
of weight $2k$ on $\Gamma$ has a basis 
$\{G_{4}^\alpha G_{6}^\beta\ |\ \alpha,\beta\in\ZZ,\ \alpha,\beta\geq 0,\
  4\alpha+6\beta=2k\}$.
In this paper we will exhibit another (simpler) basis for $M_{2k}$.
It is given by  
   $\{G_{2k}\}\cup\{G_{4i}G_{2k-4i}\ |\ i=1,2,\ldots,d_k\}$
if $2k\equiv 0\pmod 4$, and
   $\{G_{2k}\}\cup\{G_{4i+2}G_{2k-4i-2}\ |\ i=1,2,\ldots,d_k\}$
if $2k\equiv 2\pmod 4$ where $d_k+1=\dim_{\CC} M_{2k}$.
\end{abstract}

\maketitle

\section{Introduction and statement of results}
\label{sect1}

Modular forms of one variable have been studied for a long time.
They appear in many areas of mathematics and in theoretical physics.
In this paper we consider the space $M_{2k}$ of modular forms 
of weight $2k$, and find a simple basis for $M_{2k}$
in terms of Eisenstein series, which is different from the classically
known standard basis. A motivation for looking for a new basis will be 
explained below.

Throughout the paper, we use the following notation:
\begin{align*}
  k &\text{\ is an integer greater than or equal to }1, \\
  \Gamma&:=SL_2(\ZZ) \text{\ \ (the full modular group)}, \\
  M_{2k}&:=
    \text{the $\CC$-vector space of modular forms of weight $2k$ on $\Gamma$}, \\
  S_{2k}&:=
    \text{the $\CC$-vector space of cusp forms of weight $2k$ on $\Gamma$}, \\
  S_{2k}^*&:=
    \Hom_\CC( S_{2k},\CC) \text{\ \ (the dual space of $S_{2k}$)}, \\
  d_k&:=\begin{cases}
       \lfl\frac{k}{6}\rfl-1
         & \mathrm{if\ \ \ } 2k\equiv 2\pmod {12} \\
       \lfl\frac{k}{6}\rfl
         & \mathrm{if\ \ \ } 2k\not\equiv 2\pmod {12}
       \end{cases}
  \end{align*}
where $\lfloor x\rfloor$ denotes the greatest integer not exceeding $x\in\RR$.
We note that
$$\dim_{\CC} S_{2k}=d_k \text{\ \ \ and\ \ \ } \dim_{\CC}  M_{2k}=d_k+1.$$
Let $B_{2n}$ is the $2n$th Bernoulli number and $\sigma_{2n-1}(m)$ is 
the $(2n-1)$th divisor function. Namely, 
  $$\sigma_{2n-1}(m):=\sum_{0<d|m}d^{2n-1}\ \ \ (n\geq 1).$$
Then the Eisenstein series of weight $2n$ for $\Gamma$ is defined by 
  $$G_{2n}(z):=-\frac{B_{2n}}{4n}
    +\sum_{m=1}^{\infty}\sigma_{2n-1}(m)e^{2\pi imz} $$
where $z\in \HH:=\{z\in\CC\ |\ \Im(z)>0\}$.

The classically well-known basis for $M_{2k}$ is the following set
(Serre \cite[p.\ 89]{SE2}):
  $$\{G_{4}^\alpha G_{6}^\beta\ |\ \alpha,\beta\in\ZZ,\ \alpha,\beta\geq 0,\
  4\alpha+6\beta=2k\}.$$
However, the Fourier coefficients of these forms are not so 
simple 
when we write down the coefficients as sums of products of divisor
functions.
This will motivate us to look for a new simpler basis for $M_{2k}$,
consisting of modular forms 
whose Fourier coefficients are convolution sums of two divisor functions.
Our result is formulated in the following theorem:
\begin{thm}\label{thm1.1}
\begin{enumerate}
\item
If $2k\equiv 0\pmod 4$ then
\begin{equation*}
   \{G_{2k}\}\cup\{G_{4i}G_{2k-4i}\ |\ i=1,2,\ldots,d_k\}
\end{equation*}
form a basis for $M_{2k}$.
\item
If $2k\equiv 2\pmod 4$ then
\begin{equation*}
  \{G_{2k}\}\cup\{G_{4i+2}G_{2k-4i-2}\ |\ i=1,2,\ldots,d_k\}
\end{equation*}
form a basis for $M_{2k}$.
\end{enumerate}
\end{thm}
Note that the $n$th Fourier coefficients of $G_{4i}G_{2k-4i}$ is
  $$\sum_{l=0}^{n}\sigma_{4i-1}(l)\sigma_{2k-4i-1}(n-l)$$
where we set $\sigma_{2n-1}(0):=-B_{2n}/(4n)$ by convention.

We will also find a new basis for the space of cusp forms on $\Gamma$
in the following theorem:
\begin{thm}\label{thm1.2}
\begin{enumerate}
\item
If $2k\equiv 0\pmod 4$ then
\begin{equation*}
   \{G_{4i}G_{2k-4i}+\frac{B_{4i}}{4i}\frac{B_{2k-4i}}{2k-4i}
     \frac{k}{B_{2k}}G_{2k}\ |\ i=1,2,\ldots,d_k\}
\end{equation*}
form a basis for $S_{2k}$.
\item
If $2k\equiv 2\pmod 4$ then
\begin{equation*}
  \{G_{4i+2}G_{2k-4i-2}+\frac{B_{4i+2}}{4i+2}\frac{B_{2k-4i-2}}{2k-4i-2}
     \frac{k}{B_{2k}}G_{2k}\ |\ i=1,2,\ldots,d_k\}
\end{equation*}
form a basis for $S_{2k}$.
\end{enumerate}
\end{thm}
We note that, for $\Gamma=\Gamma_0(2)$, similar but slightly different 
formulas were given in \cite[Theorem 1.6]{FY1}.

\begin{exam}
For $M_{36}$, we have a basis
  $$\{G_{36},\ G_{4}G_{32},\ G_{8}G_{28},\ G_{12}G_{24}\},$$
and for $S_{36}$,
\begin{align*}
  &\{\ G_{4}G_{32}-\frac{1479565184909325423}{286310154497221833818240}G_{36},
  \ G_{8}G_{28}-\frac{651138973032093}{122102860006168135010720}G_{36},\\
  &\ccccsp \csp \ 
  \ G_{12}G_{24}-\frac{114819293577343}{1149451061437375891652640}G_{36}\} 
\end{align*}
is a basis.
\end{exam}

\section{Preliminaries}
\label{sect2}
Let $f$ be an element of $S_{2k}$.
We write $f$ as a Fourier series
\begin{equation*}
  f(z)=\sum_{l=1}^{\infty}a_le^{2\pi ilz}.
\end{equation*}
Let $L(f,s)$ be the L-series of $f$.
Namely $L(f,s)$ is the analytic continuation of
\begin{equation*}
  \sum_{l=1}^{\infty}\frac{a_l}{l^s}
  \ \ (\Re(s)\gg 0).
\end{equation*}
Then $n$th period of $f$, \ $r_{n}(f)$,\  is defined by
\begin{equation*}
  r_{n}(f):=\int_{0}^{i\infty}f(z)z^{n}dz
             =\frac{n!}{(-2\pi i)^{n+1}}L(f,n+1)\ \ (n=0,1,\ldots,w).
\end{equation*}
Each period $r_{n}$ can be regarded as a linear map
from $S_{2k}$ to $\CC$, that is,
\begin{equation*}
  r_{n}\in S_{2k}^*=\Hom_\CC(S_{2k},\CC).
\end{equation*}

Here we recall the result of Eichler \cite{EI1}, Shimura \cite{SH1}
and Manin \cite{MA2}:
\begin{thm}[Eichler-Shimura-Manin]\label{thm2.1}
The maps
\begin{align*}
  r^+:S_{2k}&\ \ \to\ \ \CC^{k} \\
     \ f\ \ \ &\ \ \mapsto\ \
     (r_{0}(f),\ r_{2}(f),\ \ldots,\ r_{2k-2}(f))
\end{align*}
and
\begin{align*}
  r^-:S_{2k}&\ \ \to\ \ \CC^{k-1} \\
     \ f\ \ \ &\ \ \mapsto\ \
     (r_{1}(f),\ r_{3}(f),\ \ldots,\ r_{2k-3}(f))
\end{align*}
are both injective.

In other words,
\begin{enumerate}
\item
the even periods
\begin{equation*}
  r_{0},\ r_{2},\ \ldots,\ r_{2k-2}
\end{equation*}
span the vector space $S_{2k}^*$;
\item
the odd periods
\begin{equation*}
  r_{1},\ r_{3},\ \ldots,\ r_{2k-3}
\end{equation*}
also span $S_{2k}^*$.
\end{enumerate}
\end{thm}

However, these periods are not linearly independent.
A natural question was raised in \cite{FU5}: 
which periods form a basis for $S_{2k}^*$ ?
A satisfactory answer was obtained in the same paper \cite{FU5}.

To state the result in \cite{FU5} we need 
the following notation and convention:
\begin{defn}\label{defn2.1}
For an integer $i$ such that $1\leq i\leq d_k$, let
\begin{equation*}
  4i\pm 1:=\begin{cases}
       4i+1 & \mathrm{if\ \ \ } 2k\equiv 2\pmod 4 \\
       4i-1 & \mathrm{if\ \ \ } 2k\equiv 0\pmod 4.
       \end{cases}
\end{equation*}
\end{defn}

Now we can state our result in \cite{FU5}:
\begin{thm}[\cite{FU5}]\label{thm2.2}
\begin{equation*}
  \{r_{4i\pm1}\ |\ i=1,2,\ldots,d_k\}
\end{equation*}
form a basis for $S_{2k}^*$.
\end{thm}

Next we will display a basis for $S_{2k}$.
For $f,\ g\in S_{2k}$,
let $(f,g)$ denote
the Petersson scalar product.
Then there is a cusp form $R_{n}$, which is characterized
by the formula:
\begin{equation*}
  r_{n}(f)=(R_{n},f)\ \mbox{\ \ for any \ \ } f\in S_{2k}.
\end{equation*}

Passing to the dual space, we obtain a basis for $S_{2k}$.
\begin{thm}[\cite{FU5}]\label{thm2.3}
\begin{equation*}
  \{R_{4i\pm1}\ |\ i=1,2,\ldots,d_k\}
\end{equation*}
form a basis for $S_{2k}$.
\end{thm}
This theorem will be needed to prove Theorem \ref{thm1.1}.
Finally some remark on the Petersson scalar product might be in order.
\begin{rem}\label{rem2.1}
{\rm   
Let $f$ and $g$ be modular forms in $M_{2k}$
with at least one of them a cusp form. 
Then the Petersson scalar product $(f,g)$ is defined by
  $$(f,g)=\int_{\Gamma\slash \HH} f(z)\overline{g(z)}y^{2k-2}dxdy$$
where $z=x+iy$.
We note that the Petersson scalar product of an Eisenstein series 
and a cusp form is always zero (refer to \cite[p. 183]{DS1}). 

However, there is a natural extension of the Petersson scalar product
from the space of cusp forms to the space of all modular forms
(Zagier \cite[pp.\ 434--435]{ZA2}). This extended scalar product is always
non-degenerate, and furthermore, it is positive definite if and only if 
$2k\equiv 2\pmod 4$.

Petersson scalar products considered in this article are those of
extended one in the above sense which are always non-degenerate.
}
\end{rem}
\section{Proof of Theorems \ref{thm1.1} and \ref{thm1.2}}
\label{sect3}
In this section, we will give proofs of Theorems \ref{thm1.1} 
and \ref{thm1.2}.
We need the following lemma:
\begin{lem}\label{lem3.1}
Let $V$ be a $\CC$-vector space of dimension n and 
\begin{equation*}
  B:V\times V\to\CC
\end{equation*}
be a non-degenerate bilinear form.
Let 
\begin{equation*}
  \{u_{i}\in V\ |\ i=1,\ldots,n\} \text{\ \ \ and\ \ \ } 
  \{v_{i}\in V\ |\ i=1,\ldots,n\}
\end{equation*}
be two sets of vectors in $V$. Then the determinant
\begin{equation*}
  |B(u_i,v_j)|_{i,j=1,2,\ldots,n}\ne0
\end{equation*}
if and only if both $\{u_{i}\in V\ |\ i=1,\ldots,n\}$ and
$\{v_{i}\in V\ |\ i=1,\ldots,n\}$ are sets of linearly independent
vectors.
\end{lem}
The proof of this lemma is quite standard and we omit it.

\begin{proof}[Proof of Theorem \ref{thm1.1}]
First we assume that $2k\equiv 0\pmod 4$.
We consider two sets of modular forms :
\begin{equation*}
  \{G_{2k}\}\cup\{G_{4i}G_{2k-4i} \ |\ i=1,2,\ldots,d_k\} \text{\ \ \ and\ \ \ } 
  \{G_{2k}\}\cup\{R_{4i-1} \ |\ i=1,2,\ldots,d_k\}.
\end{equation*}
We would like to verify that $G_{2k},\ G_{4i}G_{2k-4i} \ (i=1,2,\ldots,d_k)$
are linearly independent. By virtue of Lemma \ref{lem3.1} it is sufficient
to show that the determinant
\vsp
\begin{equation}\label{eqn3.0}
  \begin{vmatrix}
  (G_{2k},G_{2k}) & (R_{4-1},G_{2k}) & \cdots 
    &(R_{4d_k-1},G_{2k}) \\
  (G_{2k},G_{4}G_{2k-4}) & (R_{4-1},G_{4}G_{2k-4}) 
    & \cdots &(R_{4d_k-1},G_{4}G_{2k-4}) \\
  \cdots & \cdots & \cdots & \cdots \\
  (G_{2k},G_{4d_k}G_{2k-4d_k}) & (R_{4-1},G_{4d_k}G_{2k-4d_k}) 
    & \cdots 
    &(R_{4d_k-1},G_{4d_k}G_{2k-4d_k}) \\
  \end{vmatrix}
\ne 0.
\end{equation}
\vvsp
Since $(G_{2k},G_{2k})\ne 0$ and $(R_{4i-1},G_{2k})=0$ as mentioned in 
Remark \ref{rem2.1}, \eqref{eqn3.0} is equivalent to 
\vsp
\begin{equation}\label{eqn3.1}
  \begin{vmatrix}
  (R_{4-1},G_{4}G_{2k-4}) & (R_{8-1},G_{4}G_{2k-4}) & \cdots 
    &(R_{4d_k-1},G_{4}G_{2k-4}) \\
  (R_{4-1},G_{8}G_{2k-8}) & (R_{8-1},G_{8}G_{2k-8}) & \cdots 
    &(R_{4d_k-1},G_{8}G_{2k-8}) \\
  \cdots & \cdots & \cdots & \cdots \\
  (R_{4-1},G_{4d_k}G_{2k-4d_k}) & (R_{8-1},G_{4d_k}G_{2k-4d_k}) & \cdots 
    &(R_{4d_k-1},G_{4d_k}G_{2k-4d_k}) \\
  \end{vmatrix}
\ne 0.
\end{equation}
\vvsp

Now let $\{f_i\ |\ i=1,2,\ldots,d_k\}$ be a basis for $S_{2k}$ such 
that each $f_i$ is a normalized Hecke eigenform. Then, since  
 $\{R_{4i-1}\ |\ i=1,2,\ldots,d_k\}$
is also a basis for $S_{2k}$ by Theorem \ref{thm2.3},
we know that \eqref{eqn3.1} is equivalent to 
\vsp
\begin{equation}\label{eqn3.2}
  \begin{vmatrix}
  (f_{1},G_{4}G_{2k-4}) & (f_{2},G_{4}G_{2k-4}) & \cdots 
    &(f_{d_k},G_{4}G_{2k-4}) \\
  (f_{1},G_{8}G_{2k-8}) & (f_{2},G_{8}G_{2k-8}) & \cdots 
    &(f_{d_k},G_{8}G_{2k-8}) \\
  \cdots & \cdots & \cdots & \cdots \\
  (f_{1},G_{4d_k}G_{2k-4d_k}) & (f_{2},G_{4d_k}G_{2k-4d_k}) & \cdots 
    &(f_{d_k},G_{4d_k}G_{2k-4d_k}) \\
  \end{vmatrix}
\ne 0.
\end{equation}
\vvsp
To show \eqref{eqn3.2},
we use the following Rankin's identity
(\cite{RA2}, also refer to Kohnen-Zagier \cite{KZ1} noting that
their notation of $r_n(f)$ differs from ours by a factor $i^{n+1}$):
for a normalized eigenform $f$ in $S_{2k}$,
\begin{equation}\label{eqn3.3}
\left(f,G_{2n}G_{2k-2n}
\right)
=\frac{1}{(2i)^{2k-1}}r_{2k-2}(f)r_{2n-1}(f)
\end{equation}
where $n=2,3,\ldots,k-2$.
From this identity we know that \eqref{eqn3.2} 
is equivalent to
\vsp
\begin{equation}\label{eqn3.4}
  \frac{r_{2k-2}(f_1)r_{2k-2}(f_2)\cdots r_{2k-2}(f_{d_k})}{(2i)^{(2k-1)d_k}}
  \begin{vmatrix}
  r_{4-1}(f_1) & r_{4-1}(f_2) & \cdots & r_{4-1}(f_{d_k}) \\
  r_{8-1}(f_1) & r_{8-1}(f_2) & \cdots & r_{8-1}(f_{d_k}) \\
  \cdots & \cdots & \cdots & \cdots \\
  r_{4d_k-1}(f_1) & r_{4d_k-1}(f_2) & \cdots & r_{4d_k-1}(f_{d_k}) \\
  \end{vmatrix}
\ne 0.
\end{equation}
\vvsp
Finally, \eqref{eqn3.4} is equivalent to
\vsp
\begin{equation}\label{eqn3.5}
  \begin{vmatrix}
  (R_{4-1},f_1) & (R_{4-1},f_2) & \cdots &(R_{4-1},f_{d_k}) \\
  (R_{8-1},f_1) & (R_{8-1},f_2) & \cdots &(R_{8-1},f_{d_k}) \\
  \cdots & \cdots & \cdots & \cdots \\
  (R_{d_k-1},f_1) & (R_{d_k-1},f_2) & \cdots & (R_{4d_k-1},f_{d_k}) \\
  \end{vmatrix}
\ne 0.
\end{equation}
\vvsp
Now \eqref{eqn3.5} holds, since both $\{f_i\ |\ i=1,2,\ldots,d_k\}$ 
and $\{R_{4i-1}\ |\ i=1,2,\ldots,d_k\}$ are bases for $S_{2k}$.
This implies the assertion (1) of Theorem \ref{thm1.1}.

Next we assume that $2k\equiv 2\pmod 4$. The argument similar to the above 
proves the assertion (2) of Theorem \ref{thm1.1}. 
This completes the proof.
\end{proof}

\begin{proof}[Proof of Theorem \ref{thm1.2}]
In Theorem \ref{thm1.1} we proved that
\begin{equation*} 
  \{G_{2k}\}\cup\{G_{4i}G_{2k-4i} \ |\ i=1,\ldots,d_k\}
\end{equation*}
is a basis for $M_{2k}$ and, in particular, the members are 
linearly independent. Hence 
$\{G_{2k}\}\cup\{G_{4i}G_{2k-4i}+\frac{B_{4i}}{4i}\frac{B_{2k-4i}}{2k-4i}
     \frac{k}{B_{2k}}G_{2k} \ |\ i=1,\ldots,d_k\}$
are linearly independent. This implies
$\{G_{4i}G_{2k-4i}+\frac{B_{4i}}{4i}\frac{B_{2k-4i}}{2k-4i}
     \frac{k}{B_{2k}}G_{2k} \ |\ i=1,\ldots,d_k\}$
are again linearly independent. Moreover, since 
$G_{4i}G_{2k-4i}+\frac{B_{4i}}{4i}\frac{B_{2k-4i}}{2k-4i}
     \frac{k}{B_{2k}}G_{2k}\in S_{2k}\ (i=1,\ldots,d_k)$,
these form a basis for $S_{2k}$. This completes the proof.
\end{proof}


\end{document}